\documentclass[journal,11pt,draftcls,onecolumn]{IEEEtran}
\usepackage[dvips]{graphicx}
\usepackage{boxedminipage}
\usepackage{color}
\usepackage{amsfonts}
\usepackage{amssymb}
\usepackage{amsmath}
\usepackage{graphics,amsmath,amsfonts,amssymb,epsfig,latexsym,amsfonts,graphicx,amssymb}
\usepackage{psfrag}
\usepackage{ amssymb }
\usepackage{bm}

\newcommand{\bx}{{\bm x}}
\newcommand{\bd}{{\bm d}}
\newcommand{\be}{{\bm e}}
\newcommand{\bu}{{\bm u}}
\newcommand{\bv}{{\bm v}}
\newcommand{\by}{{\bm y}}
\newcommand{\bz}{{\bm z}}
\newcommand{\prox}{{\rm prox}}

\newtheorem{theorem}{Theorem}[section]
\newtheorem{lemma}[theorem]{Lemma}
\newtheorem{proposition}[theorem]{Proposition}
\newtheorem{corollary}[theorem]{Corollary}

\newtheorem{definition}[theorem]{Definition}
\newenvironment{proof}[1][Proof]{\begin{trivlist}
\item[\hskip \labelsep {\bfseries #1}]}{\end{trivlist}}

\title{On the Linear Convergence of the Approximate Proximal Splitting Method for Non-Smooth Convex Optimization}
\author{Mojtaba Kadkhodaie, Maziar Sanjabi and Zhi-Quan Luo}
\begin{document}
\maketitle
\begin{abstract}
Consider the problem of minimizing the sum of two convex functions, one being smooth and the other non-smooth. In this paper, we introduce a general class of approximate proximal splitting (APS) methods for solving such minimization problems. Methods in the APS class include many well-known algorithms such as the proximal splitting method (PSM), the block coordinate descent method (BCD) and the approximate gradient projection methods for smooth convex optimization. We establish the linear convergence of APS methods under a local error bound assumption. Since the latter is known to hold for compressive sensing and sparse group LASSO problems, our analysis implies the linear convergence of the BCD method for these problems without strong convexity assumption.
\end{abstract}

\section{Introduction}\label{seq:Introduction}
Consider a constrained convex minimization of the form
\begin{align}
\min_{\bx\in {X}}\quad F(\bx)=f_1(\bx)+f_2(\bx),\label{eq:main}
\end{align}
where ${X}$ is a convex closed set (can be $\mathbb{R}^n$), $f_1$ is a convex function (may be non-smooth) and $f_2$ is a smooth convex function with Lipschitz continuous gradient.
\subsection{Motivating Applications}
Non-smooth convex optimization problems of the form \eqref{eq:main} arise in many contemporary statistical and signal processing applications including compressive sensing, signal denoising and sparse logistic regression. In the sequel, we outline some of the most recent applications of problem \eqref{eq:main}.\\

{\it LASSO Problem:} Suppose that we have a noisy observation vector $\bm b\in \mathbb{R}^m$ about an unknown sparse vector $\bx\in \mathbb{R}^n$, where the signal model is linear and given by
\begin{align}
{\bm b} \approx {\bm A\bx},\nonumber
\end{align}
for some given matrix $\bm A\in \mathbb{R}^{m \times n}$. One of the most popular techniques to estimate the sparse vector $ \bm x$ is called LASSO \cite{Tibshirani96}. LASSO can be viewed as an $\ell_1$-norm regularized linear least squares problem
\begin{align}
\min_{\bx\in \mathbb{R}^n} \frac{1}{2}\|{\bm A}\bx-{\bm b}\|^2+\lambda \|\bx\|_1,\label{eq:LASSO}
\end{align}
where the first term $\frac{1}{2}\|\bm A\bx-\bm b\|^2$ reduces the estimation error, and the second term $\lambda \| \bm x \|_{1}$ promotes the sparsity of the solution. The parameter $\lambda$ controls the sparsity level of the solution. The higher $\lambda$ is, the fewer non-zero entries would be in the LASSO solution. Clearly, by setting $f_2(\bm x)=\frac{1}{2}\|\bm A\bx-\bm b\|^2$, $f_1(\bm x)=\lambda \|\bx\|_1$ and ${X}=\mathbb{R}^n$, problem \eqref{eq:LASSO} becomes a special case of problem \eqref{eq:main}.



{\it Group LASSO Problem:} In many applications, such as image denoising, the desired solution should have the so called \emph{group sparse}
structure \cite{groupLASSO06}, i.e. the solution $\bx$ should admit a group separable structure $\bx=[\bx_{1},...,\bx_{K}]^{'}$, where $\bx_{i}\in\mathbb{R}^{n_{i}}$ and $\sum_{i=1}^{K} n_{i}=n$, with only a few non-zero representing groups. In these applications, the following optimization problem needs to be solved
\begin{align}
\min_{\bx\in \mathbb{R}^n} \frac{1}{2}\|{\bm A\bx-\bm b}\|^2+\sum_{J\in \mathcal{J}} w_{J}\|\bx_J\|, \label{eq:GLASSO}
\end{align}
where $\mathcal{J}$ is a partition of $\{1,\cdots,n\}$ and $w_{J}$ is the sparsity weight of block $J$.
Setting  $f_{2}(\bm x)=\frac{1}{2}\|{\bm A\bx-\bm b}\|^2$ and $f_{1}(\bm x) =\sum_{J\in \mathcal{J}} w_{J}\|\bx_J\|$, the Group LASSO problem \eqref{eq:GLASSO} follows the structure of problem \eqref{eq:main}.

{\it Group LASSO for Logistic Regression:} Given a set of $n$-dimensional feature vectors $\bm a_i$, $i=1,\cdots,m$, and the corresponding class labels $b_i\in\{0,1\}$, $i=1,\cdots, m$, our task is to find a linear classifier for the vectors $\bm a_{i}$. Assume the probability distribution of the class label $b$, given a feature vector $\bm a$ is given by
\begin{align}
p(b=1 |\bm a;\bx)=\frac{\exp(\bm a^T\bx)}{1+\exp(\bm a^T\bx)},\nonumber
\end{align}
where  $\bm x$ is the logistic coefficient vector.
The logistic Group LASSO problem \cite{GLASSOLogistic} can be written as
\begin{align}
\min_{\bx\in \mathbb{R}^n} \sum_{i=1}^m (\log(1+\exp(\bm a_i^T\bx))-b_i\bm a_i^T\bx)+\sum_{J\in \mathcal{J}}w_J\|\bx_J\|,\label{eq:LogReg}
\end{align}
where $w_{J}$ is the sparsity weight for the corresponding block $\bm x_{J}$. Problem \eqref{eq:LogReg} can also be interpreted as a special form of problem \eqref{eq:main}.
We refer the readers to \cite{Bach08,MA07} for further applications of Group LASSO, and to \cite{LearningKim,LearningLiu,LearningRoth,LearningBerg, Wright09,GLASSOLogistic} for further studies on Group LASSO type of techniques in statistical problems.
\subsection{First Order Methods to Solve Convex Problem \eqref{eq:main}}

For the large scale optimization problems of the form \eqref{eq:main}, the preferred approach is to use iterative descent algorithms along gradient related directions. Such first order methods have a long history in optimization. For example, if we assume ${X}=\mathbb{R}^n$ and $f_1(\cdot)$ to be the indicator function of a closed convex set $\bm C$, then the problem in \eqref{eq:main} turns out to be the smooth minimization of $F(\bx)=f_2(\bx)$ over the set $\bm C$. A well-known first order method to solve this problem is called Gradient Projection (GP) \cite{Rosen60,Rosen61} algorithm. In iteration $k$, the GP algorithm takes a gradient step of size $\alpha_{k}$ and then projects the point back into the feasible set $\bm C$,
\begin{align}
\bx^{k+1}=\mbox{Proj}_{\bm C} [\bx^k-\alpha_k\nabla f_2 (\bx^{k})].\label{eq:GP}
\end{align}

The convergence analysis of the GP method has been studied before \cite{Luo93}. It has been shown that such analysis can be generalized to approximate versions of the GP method \cite{Luo93,Luo92,Luo94,Mangasarian91,Li93} which are also known as Approximate Gradient Projection (AGP) methods. In the framework of AGP, an error is allowed in the computation of gradient, as long as the size of the error vector is
sufficiently small (see Section \ref{section:APS}). In particular, the update of an AGP algorithm can be defined as
\begin{align}
\bx^{k+1}=\mbox{Proj}_{\bm C}[\bx^{k}-\alpha_{k}\nabla f_{2}(\bx^{k})+\be^k],\label{eq:AGPupdate}
\end{align}
where $\|\be^k\|\leq \kappa \|\bx^{k+1}-\bx^{k}\|$ for some $\kappa>0$.
It has been shown \cite{Luo93} that many well-known algorithms such as Matrix Splitting Method \cite{Pang82} and Extragradient Method \cite{Kor76} are all special cases of the AGP algorithm \eqref{eq:AGPupdate}. We will extend this approach to the non-smooth optimization \eqref{eq:main}; see Section \ref{section:APS}.

The analog of the GP algorithm for the general non-smooth version of problem  \eqref{eq:main} is the so called \textit{Proximal Splitting Method}  (PSM). In order to introduce this method, we first need to define the proximity operator. For any convex function $\varphi(\cdot)$ (possibly non-smooth), the Moreau-Yashida proximity operator ${\rm prox}_{\varphi}(\cdot, X):~\mathbb{R}^{n}\rightarrow \mathbb{R}^{n}$ is defined as
\begin{align}
 {\rm prox}_{\varphi} (\bx, X)=\arg\min_{\bm y\in X}\varphi(\bm y)+\frac{1}{2}\|\bm y-\bx\|^{2}. \label{eq:prox}
\end{align}
Note that since $\frac{1}{2}\|\cdot-\bx\|^{2}$ is strongly convex and $\varphi(\cdot)$ is convex, the minimizer of \eqref{eq:prox} is unique. Furthermore, if the function $\varphi$ is chosen to be the indicator function, $\iota_{C}$, of the closed convex set $C$ and $X=\mathbb{R}^{n}$, then the proximity operator is reduced to the projection operator onto the set $C$.  Thus, the proximity operator is a natural extension of the projection operator. In the sequel, we will denote the proximity operator simply by $\rm{prox}_{\varphi}(\cdot)$ and assume that its dependence on the set $X$ is understood from the context.

The proximity operator inherits many useful properties of the projection operator into convex sets. For instance, the proximity operator is known to be non-expansive and is therefore Lipschitz continuous,
\begin{align}
\|{\rm prox}_{\varphi}(\bx)-{\rm prox}_{\varphi}(\bm y)\|\leq \|\bm x-\bm y\|,~\forall~ \bx,\bm y \in \mathbb{R}^{n}.\nonumber
\end{align}
 For large scale problems, it is not always easy to compute the proximity operator, unless the function $\varphi$ has some special structure, such as separability. In those cases the proximity operator is efficiently computable (or has closed form). For instance, if the function $\varphi$ is the $\ell_{1}$-norm, the proximity operator has a closed form solution, also known as the Shrinkage operator \cite{Tseng09}.

Using the proximity operator, the optimality condition of problem \eqref{eq:main} can be formulated as
\begin{align}
\bx={\rm prox}_{\alpha f_{1}}(\bx-\alpha \nabla f_{2}(\bx)),\label{eq:optimality}
\end{align}
for some $\alpha>0$. The Proximal Splitting Method (PSM) can be viewed as an iterative approach to solve the above fixed point equation
\begin{align}
\bx^{k+1}={\rm prox}_{\alpha_{k}f_{1}}(\bx^{k}-\alpha_{k}\nabla f_{2}(\bx^{k})), \label{eq:PSM}
\end{align}
where $\alpha_{k}>0$ determines the step size at iteration $k$. Note that PSM is identical to the GP method, if $f_{1}=\iota_{C}$ for some convex closed set $C$.
It is known that if the step size $\alpha_{k}$ satisfies
\begin{align}
0<\underline \alpha\leq \alpha_{k}\leq \bar \alpha <1/L, ~k=0,1,\cdots\nonumber
\end{align}
 where $L$ is the Lipschitz constant of $\nabla f_{2}$, then every sequence generated by the PSM converges to a solution of \eqref{eq:main} (see  \cite{Combettes05}).

In spite of this convergence result, the rate of convergence for the PSM is not known, except for some specific cases. For instance, if $f_{1}=\iota_{C}$ and $f_{2}$ has a composite structure ($f_{2}(x)=h(\bm A\bx)$, where $h$ is strongly convex and $\bm A$ is an arbitrary $m\times n$ matrix), then it is proved by Luo and Tseng \cite{Luo92} that the PSM algorithm (which coincides with GP in this case) converges linearly to an optimal solution of \eqref{eq:main}. This result establishes linear convergence in the absence of strong convexity of the objective function. This result has been recently extended to the case where, $f_{1}(x)=\sum_{J\in \mathcal{J}}w_{J}\|\bx_{J}\|_{2}$ and $f_{1}(x)=\sum_{J\in \mathcal{J}}w_{J}\|\bx_{J}\|_{2}+\lambda\|\bx\|$, where $w_J$ and $\lambda$ are some nonnegative constants, and $f_{2}=h(\bm A \bx)$ is still a composite function (see \cite{Tseng09,Haibin12}).


\subsection{Our Contribution}
This paper proposes a general framework of first order methods, called \emph{approximate proximal splitting} (APS), for the nonsmooth convex optimization problem \eqref{eq:main}. This framework combines the existing framework of AGP with the proximal splitting technique, and as such, it includes the GP, AGP and proximal splitting methods as special cases. Moreover, the well known block coordinate descent (BCD) algorithm is also a special case of APS.

We analyze the convergence rate of APS class of algorithms under a local error bound condition. Our result implies the linear convergence rate of Block Coordinate Descent Method (BCD) for \eqref{eq:main} for the LASSO or group LASSO type of problems  when $f_{1}(\bx)=\sum_{J\in \mathcal{J}}w_{J}\|\bx_{J}\|_{2}$ or $f_{1}(\bx)=\sum_{J\in \mathcal{J}}w_{J}\|\bx_{J}\|_{2}+\lambda \|\bx\|_{1}$.
The BCD algorithm is one of the main algorithms used to solve large scale optimization problems due to the simplicity of its updates (especially for the LASSO or group LASSO type of problems in which each step of BCD is equivalent to a shrinkage operator \cite{Haibin12}). This linear convergence result provides theoretical proof for effectiveness of BCD in handling such problems.

Our result differs from the existing proximal splitting methods and analysis in several aspects. Among the existing work  \cite{agarwal2011fast, Haibin12, schmidt2011convergence}, the only one which considers an error term in the proximal splitting algorithm is \cite{schmidt2011convergence}, while  the other two (\cite{agarwal2011fast} and \cite{Haibin12}) are focussed on the pure proximal splitting algorithm. The result in \cite{schmidt2011convergence} does not provide the linear convergence except in the strongly convex case which is a special case of our result, so it cannot be used to establish the linear convergence of the BCD algorithm for LASSO type problems. The reason for such difference is that we use a local error bound condition in place of the strong convexity assumption.

The result in \cite{agarwal2011fast} deals with the problem of linear convergence from a statistical point of view. It assumes that problem \eqref{eq:main} comes from an $M$-estimator formulation with some probabilistic construction. It proves that the iterates will converge linearly to a neighborhood around the optimal solution, but not necessarily an optimal solution. As such, this result is probabilistic and not deterministic. This is in contrast to our result which is a general convex optimization problem  in the form \eqref{eq:main}, regardless how it is generated.
That said, by utilizing the so called restricted strong convexity and restricted smoothness (see \cite{agarwal2011fast}) instead of an error bound, the authors have established the linear convergence of the proximal splitting algorithm for a broad range optimization problems with non-smooth regularizers such as $L_1$ norm or nuclear norm.

\section{Preliminaries}
\subsection{Proximal Gradient Vector}
Before introducing the APS method, we need to define some basic concepts, which will be useful in our future convergence analysis of the algorithm.
\begin{definition}For any $\alpha>0$, we define \textit{proximal gradient} vector as
\begin{align}
\tilde \nabla F(\bx,\alpha)=\frac{1}{\alpha}[\bx-{\rm prox}_{\alpha f_{1}}(\bx-\alpha \nabla f_{2}(\bx))].\label{eq:PSgrad}
\end{align}
When $\alpha=1$, we will use the short notation
\begin{equation}
\tilde \nabla F(\bx)=\tilde \nabla F(\bx,\alpha).
\end{equation}
\end{definition}
Note that in the special case of $f_{1}=0$ and $X=\mathbb{R}^{n}$, the proximal gradient reduces to the standard gradient, namely, $\tilde \nabla F(\bx,\alpha)=\nabla f_{2}(\bx)=\nabla F(\bx)$. In another special  case where $f_{1}=\iota_{C}$ (the indicator function of a convex set $C$), we have
\begin{equation}
\tilde \nabla F (\bx,\alpha)=\frac{1}{\alpha}[\bx-{\rm proj}_{C}(\bx-\alpha\nabla f_{2}(\bx))],
\end{equation}
which is the residual of the optimality condition for the following problem
\begin{equation}
\min_{\bx \in C} F(\bx)=f_2(\bx).
\end{equation}
Hence, $\tilde \nabla F(\bx,\alpha)$ can be viewed as a generalized notion of gradient for the constrained non-smooth minimization. In addition, it inherits many useful properties of gradient. For instance, $\tilde \nabla F(\bx^{*},\alpha)=0$ for some $\alpha>0$ iff $\bx^{*}$ is an optimal solution of \eqref{eq:main}.

The optimality condition for \eqref{eq:main} given in \eqref{eq:optimality} suggests that we can define a local measure for distance to optimality by
\begin{align}
\psi(\bx)=\| \tilde \nabla F(\bx)\|=\|\bx-\prox_{f_{1}}(\bx-\nabla f_{2}(\bx))\| \label{eq:residual}.
\end{align}
It is easy to see that $\psi (\bx)=0$ iff $\bx$ belongs to the set of optimal solutions of \eqref{eq:main}, which we denote by $X^{*}$.


\subsection{Error Bounds}\label{section:errorbound}
In this section we formally introduce the notion of error bound. As we will see it is a vital property in obtaining linear convergence rate for solving a problem via first-order methods.

For any $x\in X$, we can define
\begin{align}
\varphi (\bx)=\min_{\by\in \bar X^{*}}\|\bx-\by\|,\label{eq:Optdist}
\end{align}
where $\bar X^{*}$ is the closure of $X^{*}$ (the set of optimal solutions of \eqref{eq:main}). It is straightforward to see that $\varphi (\bx)$ can be used as a measure for distance to optimality, and $\varphi (\bx)=0$ iff $\bx\in \bar X^{*}$. However, in practice it is impossible to compute $\varphi(\bx)$, due to the requirement of knowing the set of optimal solutions, $\bar X^{*}$. This is where the error bound comes into the picture. It serves as an approximated measure of the distance to optimality. The error bound is simply a bound on $\varphi (\bx)$, based on another measure of optimality that can be computed easily (in this case, the size of the residual $\psi(\bx)$ defined by \eqref{eq:residual}).

\begin{definition}{ \bf {(Local Error Bound \cite{Luo93})}}\label{def:ErrB}
Consider the optimality distance measures defined by \eqref{eq:residual} and \eqref{eq:Optdist}. We say that problem \eqref{eq:main} satisfies the local error bound property if for every $\nu\geq \inf_{\bx\in X} F(\bx)$, there exist scalars $\delta>0$ and $\tau>0$ such that
\begin{align}
\varphi (\bx)\leq \tau \psi (\bx),\label{eq:ErrB}
\end{align}
for all $\bx\in X$ with $F(\bx)\leq \nu$ and $\psi(\bx)\leq \delta$.
\end{definition}

In other words, \eqref{eq:ErrB} says that $\varphi (\bx)$ is bounded above by the norm of the residual at $x$, whenever $F(\bx)$ is bounded above and this residual is small enough. In order to gain some insight on when \eqref{eq:ErrB} holds, consider the case where  $X=\mathbb{R}^{n}$ and $F(x)=\frac{1}{2}\bx^{T}\bm A\bx+\bm b^{T}\bx$ for some Positive definite matrix $\bm A$ and a vector $\bm b\in \mathbb{R}^{n}$. Then, \eqref{eq:ErrB} is equivalent to
\begin{align}
\varphi (\bx)\leq \tau \|\nabla f_{2}(\bx)\|=\tau \|\bm A\bx+\bm b\|\nonumber,
\end{align}
which can be easily checked to be true (using elementary linear algebra). 
Furthermore, the error bound holds globally for strongly convex smooth $F(\bx)$ for any closed convex set $X$. In case when $X=\mathbb{R}^{n}$ the  strong convexity of $F(\bx)$ implies the existence a $\tau>0$ such that
\begin{align}
\|\bx-\by\|^2\leq \tau \langle \nabla F(\bx)-\nabla F(\by),\bx-\by\rangle ~\forall~\bx,\by \nonumber.
\end{align}
Let $\hat \bx$ to be a stationary point in $\bar X^{*}$ satisfying $\|\bx-\hat \bx\|=\varphi (\bx)$, then
\begin{align}
\varphi(\bx)^{2}\leq \tau \langle \nabla F(\bx),\bx-\hat \bx
\rangle \leq\tau \|\nabla F(\bx)\|\|\bx-\hat \bx\|\nonumber.
\end{align}
Proving error bound for different problems has a long history in literature. It was first considered by Demb and Tulowizki \cite{DeT84} for strongly convex quadratic functions and by Pang \cite{Pang86} in the context of Linear Complementarity Problems (LCP). In the case of smooth minimization, there are numerous error bound results under different assumptions. For instance, it has been shown that the error bound holds for strongly convex functions in \cite{Pang87}, and for quadratic functions with polyhedral constraint in \cite{Luo92b,Luo92}.

In the case of non-smooth optimization, the available error bound results are far fewer and only deal with problems with structured non-smooth parts. For instance, in the recent works \cite{Haibin12, Tseng09}, it has been proved that error bounds holds for special type of non-smooth problems (Group LASSO type of problems).
\begin{theorem}\label{theorem:errorbound} \textit {(\cite{Haibin12, Tseng09})} In problem \eqref{eq:main} let $X=\mathbb{R}^{n}$, $f_{1}(\bx)=\sum_{J\in \mathcal{J}} w_{J}\|\bx_{J}\|+\lambda\|\bx\|$ for non-negative $w_{J}$'s and $\lambda$ and $f_{2}(\bx)=h(\bm A\bx)$ for some strongly convex smooth function $h:\mathbb{R}^{m}\mapsto \mathbb{R}$ and an $m\times n$ matrix $\bm A$. In addition if the function $F$ is coercive, then error bound \eqref{eq:ErrB} holds for \eqref{eq:main}.
\end{theorem}
\begin{corollary}\label{coro:EB} The error bound condition \eqref{eq:ErrB} holds for LASSO problem \eqref{eq:LASSO}, Group LASSO problem \eqref{eq:GLASSO} and logistic Group LASSO problem \eqref{eq:LogReg}.
\end{corollary}

We will utilize this result to establish our proof for linear convergence of APS class of methods for such problems.
\section{Approximate Proximal Splitting (APS) Method}\label{section:APS}
In this section we formally introduce the Approximate Proximal Splitting (APS) class of methods.
\begin{definition}{\bf(APS)} An algorithm is considered in the class of APS methods if it generates a sequence of iterates $\bx^{0},\bx^{1},\cdots$ in $X$ such that
\begin{align}
\bx^{r+1}={\rm prox}_{\alpha^{r} f_{1}}(\bx^{r}-\alpha^{r}(\nabla f_{2}(\bx^{r})+\be^{r})),~r=0,1,\cdots~,\label{eq:APS}
\end{align}
where $\{\alpha^{r}\}$ is a sequence of positive scalars with $\lim\inf \alpha^{r}>0$ and $\{\be^{r}\}$ is a sequence in $\mathbb{R}^{n}$ with
\begin{align}
\|\be^{r}\|\leq \kappa \|\bx^{r}-\bx^{r+1}\|\label{eq:ErrCond},
\end{align}
for some non-negative scalar $\kappa$.
\end{definition}

In equation \eqref{eq:APS}, $\alpha^{r}$ and $\be^{r}$ may depend on $\bx^{r}$ and can be viewed as algorithm parameters. Hence, different choices of $\alpha^{r}$ and $\be^{r}$ lead into different algorithms. For instance, the PSM algorithm whose update rule is given by \eqref{eq:PSM},  is a special case of the APS algorithm with $\bm e^{r}=\bm 0$. In fact, the condition \eqref{eq:ErrCond} ensures that the algorithm does not deviate too much from the PSM update.

For smooth minimization which is a special case of problem \eqref{eq:main} with $f_{1}=\iota_{C}$, the AGP class of algorithms is very common. Since the proximity operator reduces to the projection operator in this case, the APS algorithm \eqref{eq:AGPupdate} contains the APG method \eqref{eq:APS} as a special case. Later we will see how the Block Coordinate Decent (BCD) algorithm is also a special case of the APS method.

\section{Linear Convergence of APS}\label{sec:LinearCon}
In this section we will prove that any APS method converges linearly, if the following properties hold true.
\begin{itemize}
\item {\bf Sufficient Decrease:} There exists a constant $c_{1}>0$ such that $\forall\ r\geq0$,
\begin{align}
F(\bx^{r})-F(\bx^{r+1})\geq c_{1}\|\bx^{r+1}-\bx^{r}\|^{2}.\label{eq:SuffD}
\end{align}
\item {\bf Local Error Bound}: Definition \ref{def:ErrB}.
\item {\bf Cost-to-go:} There exists a $c_{2}>0$, such that
\begin{align}
F(\bx^{r})-F^{*}\leq c_{2}(\varphi(\bx^{r})^{2}+\|\bx^{r+1}-\bx^{r}\|^{2}),~\forall~r,\label{eq:Cost2Go}
\end{align}
where $F^{*}$ is the optimal objective value of \eqref{eq:main}.
\end{itemize}

The local error bound property solely depends on the optimization problem. Therefore the problem structure needs to be studied to ensure this property holds. In the literature, this condition has been established for certain classes of optimization problems, see \cite{Luo93}, \cite{Haibin12}, \cite{Tseng01} and references therein. Some of the existing results were summarized in Theorem \ref{theorem:errorbound}.

The section proceeds as follows. Assuming the sufficient decrease condition, we first prove that the cost-to-go will naturally follow for the APS class of algorithms. Then, the sufficient decrease is proved under some assumptions on the step size $\alpha^{r}$ and the error vector $\be^{r}$. Finally, the linear convergence rate of APS methods is shown assuming the local error bound condition of the problem.
The following lemma shows that the cost-to-go property follows from the sufficient decrease condition.


\begin{lemma} \label{lm:4.1}
If an APS method satisfies the sufficient decrease condition \eqref{eq:SuffD}, then the cost-to-go condition \eqref{eq:Cost2Go} also holds.
\end{lemma}

\begin{proof}
Set $\hat {\bm x}^{r}$ to be the point in $\bar X^{*}$, such that $\varphi({\bm x}^{r})=\|{\bm x}^{r}-\hat {\bm x}^{r}\|$. The optimality condition of ${\bm x}^{r+1}$ implies
\begin{align}
&f_{1}(\hat{\bm x}^{r})+\langle \nabla f_{2}({\bm x}^{r})+{\bm e}^{r}, \hat{\bm x}^{r}-{\bm x}^{r}\rangle+\frac{1}{2\alpha^{r}}\| \hat{\bm x}^{r}-{\bm x}^{r}\|^{2}~ \ge \nonumber \\
&f_{1}({\bm x}^{r+1})+\langle \nabla f_{2}({\bm x}^{r})+{\bm e}^{r}, {\bm x}^{r+1}-{\bm x}^{r}\rangle+\frac{1}{2\alpha^{r}}\| {\bm x}^{r+1}-{\bm x}^{r}\|^{2}\nonumber
\end{align}
This implies
\begin{equation}
\langle \nabla f_{2}({\bm x}^{r}) +{\bm e}^{r}, {\bm x}^{r+1}-\hat{\bm x}^{r}\rangle+f_{1}({\bm x}^{r+1})-f_{1}(\hat{\bm x}^{r})\le  \frac{1}{2\alpha^{r}} \varphi^{2}({\bm x}^{r}).
\end{equation}
Also, the mean value theorem implies\\
\begin{equation}
f_{2}({\bm x}^{r+1})-f_{2}(\hat{\bm x}^{r}) = \langle \nabla f_{2}({ \bm \xi}^{r}), {\bm x}^{r+1}-\hat{\bm x}^{r}\rangle,
\end{equation}
for some ${ \bm \xi}^{r}$ in the line segment joining ${\bm x}^{r+1}$ and $\hat{\bm x}^{r}$. Combining the above two relations yields
\begin{align}
F({\bm x}^{r+1})-F(\hat{\bm x}^{r})&=f_{1}({\bm x}^{r+1})-f_{1}(\hat{\bm x}^{r}) +f_{2}({\bm x}^{r+1})-f_{2}(\hat{\bm x}^{r}) \nonumber \\
&=\langle \nabla f_{2}({ \bm \xi}^{r}), {\bm x}^{r+1}-\hat{\bm x}^{r}\rangle+f_{1}({\bm x}^{r+1})-f_{1}(\hat{\bm x}^{r}) \nonumber \\
&=\langle \nabla f_{2}({ \bm x}^{r}), {\bm x}^{r+1}-\hat{\bm x}^{r}\rangle+\langle \nabla f_{2}({ \bm \xi}^{r})-\nabla f_{2}({ \bm x}^{r}), {\bm x}^{r+1}-\hat{\bm x}^{r}\rangle+f_{1}({\bm x}^{r+1})-f_{1}(\hat{\bm x}^{r}) \nonumber \\
&\le \langle \nabla f_{2}({ \bm x}^{r}), {\bm x}^{r+1}-\hat{\bm x}^{r}\rangle+L\| {\bm \xi}^{r}-{\bm x}^{r} \| \| {\bm x}^{r+1}-\hat{\bm x}^{r}\|+f_{1}({\bm x}^{r+1})-f_{1}(\hat{\bm x}^{r}) \nonumber \\
&\le \frac{1}{2\alpha^{r}}\varphi^{2}({\bm x}^{r})-\langle {\bm e}^{r},{\bm x}^{r+1}-\hat{\bm x}^{r}\rangle+L\| {\bm \xi}^{r}-{\bm x}^{r} \| \| {\bm x}^{r+1}-\hat{\bm x}^{r}\| \nonumber \\
&\le \frac{1}{2\alpha^{r}}\varphi^{2}({\bm x}^{r})+L\| {\bm \xi}^{r}-{\bm x}^{r} \| \| {\bm x}^{r+1}-\hat{\bm x}^{r}\|+\kappa  \| {\bm x}^{r+1}-{\bm x}^{r} \| \| {\bm x}^{r+1}-\hat{\bm x}^{r}\| \label{eq:derC2G}
\end{align}

It remains to bound the last two terms in \eqref{eq:derC2G}. Using the fact that ${\bm \xi}^{r}$ lies in the line segment joining ${\bm x}^{r+1}$ and $\hat{\bm x}^{r}$, it follows that
\begin{align}
\| {\bm \xi}^{r}-{\bm x}^{r} \| \| {\bm x}^{r+1}-\hat{\bm x}^{r}\| &\le (\| {\bm x}^{r+1}-{\bm x}^{r} \|+\| {\bm x}^{r}-\hat{\bm x}^{r} \|)(\| {\bm x}^{r+1}-{\bm x}^{r} \| +\| {\bm x}^{r}-\hat{\bm x}^{r} \|) \nonumber \\
&=(\| {\bm x}^{r+1}-{\bm x}^{r} \|+\varphi({\bm x}^{r}))(\| {\bm x}^{r+1}-{\bm x}^{r} \| +\varphi({\bm x}^{r})) \nonumber \\
&\le 2(\| {\bm x}^{r+1}-{\bm x}^{r} \|^{2}+\varphi^2({\bm x}^{r})).\nonumber
\end{align}
For the last term in \eqref{eq:derC2G} we have,
\begin{align}
\| {\bm x}^{r+1}-{\bm x}^{r} \| \| {\bm x}^{r+1}-\hat{\bm x}^{r}\| &\le \| {\bm x}^{r+1}-{\bm x}^{r} \| (\| {\bm x}^{r+1}-{\bm x}^{r} \| +\varphi({\bm x}^{r})))  \nonumber \\
&\le2\| {\bm x}^{r+1}-{\bm x}^{r} \| ^{2}+2\varphi^{2}(\bm x^{r}).\nonumber
\end{align}
Substituting these upper bounds into the right hand side of inequality \eqref{eq:derC2G} yields
\begin{align}
F({\bm x}^{r+1})-F(\hat{\bm x}^{r})={\textit O}(\varphi^{2}({\bm x}^{r})+\| {\bm x}^{r+1}-{\bm x}^{r} \|^{2}).
\end{align}
This proves the desired result.
\end{proof}

The following result establishes the sufficient decrease condition for the APS algorithm under some conditions on the error sequence $\be^{r}$ and the step size sequence $\alpha^{r}$.

\begin{lemma} \label{lemma: sufdec} Consider an APS algorithm defined by \eqref{eq:APS}-\eqref{eq:ErrCond} for some $\kappa>0$ and some stepsizes $\alpha^r$ satisfying
\begin{equation}\label{eq:stepsize}
0<\underline \alpha\leq \alpha^{r}\leq \bar \alpha <\frac{2}{L+2\kappa},\quad \mbox{for some $\underline \alpha$ and $\bar\alpha$, $\forall \ r$.}
\end{equation}
then the sufficient decrease property \eqref{eq:SuffD} holds.
\end{lemma}

\begin{proof}
By the optimality condition for $\bx^{r+1}$, there exists a $\bm g\in \partial f_{1}(\bx^{r+1})$ such that
\begin{align}
\langle\alpha^{r}\bm g+\alpha^{r}\nabla f_{2}(\bx^{r})+\alpha^r\be^{r}+\bx^{r+1}-\bx^{r},\by-\bx^{r+1}\rangle\geq0,~\forall\ \by\in X.\nonumber
\end{align}
Moreover, the convexity of $f_{1}$ implies that there is some $\bm g\in \partial f_{1}(\bx^{r+1})$ which satisfies the following inequality
\begin{align}
f_{1}(\by)-\langle \bm g, \by-\bx^{r+1}\rangle \geq f_{1}(\bx^{r+1}).\nonumber
\end{align}
Using the above two relations and convexity of $f_{2}$, we obtain
\begin{align}
F({\bm y})&\ge f_{1}({\bm y})+f_{2}({\bm x^{r}})+\langle \nabla f_{2}({\bm x^{r}}),{\bm y}-{\bm x^{r}}\rangle  \nonumber \\
&= f_{1}({\bm y})+f_{2}({\bm x^{r}})+\langle \nabla f_{2}({\bm x}),{\bm x^{r+1}}-{\bm x}^{r}\rangle+\langle \nabla f_{2}({\bm x^{r}}),{\bm y}-{\bm x^{r+1}}\rangle \nonumber \\
&\geq f_{1}({\bm y})+f_{2}({\bm x^{r}})+\langle \nabla f_{2}({\bm x^{r}}),{\bm x^{r+1}}-{\bm x^{r}}\rangle-\langle\be^{r}+\frac{1}{\alpha^r}(\bx^{r+1}-\bx^{r})+\bm g,{\bm y}-{\bm x^{r+1}}\rangle \nonumber \\
&\geq [f_{1}({\bm x}^{r+1})+\langle {\bm g},{\bm y}-{\bm x}^{r+1} \rangle]+f_{2}({\bm x^{r}})+\langle \nabla f_{2}({\bm x^{r}}),{\bm x^{r+1}}-{\bm x^{r}}\rangle \nonumber \\
&-\langle\be^{r}+\frac{1}{\alpha^r}(\bx^{r+1}-\bx^{r})+\bm g,{\bm y}-{\bm x^{r+1}}\rangle  \nonumber \\
&= f_{1}({\bm x^{r+1}})+f_{2}({\bm x^{r}})+\langle \nabla f_{2}({\bm x^{r}}),{\bm x^{r+1}}-{\bm x^{r}}\rangle-\langle\be^{r}+\frac{1}{\alpha^r}(\bx^{r+1}-\bx^{r})+\bm g,{\bm y}-{\bm x^{r+1}}\rangle . \nonumber
\end{align}
Now the Lipschitz continuity of $\nabla f_{2}$ and Taylor expansion of $f_{2}$ imply that
\begin{align}
F(\by)\geq &  f_{1}({\bm x^{r+1}})+f_{2}({\bm x^{r}})+\langle \nabla f_{2}({\bm x^{r}}),{\bm x^{r+1}}-{\bm x^{r}}\rangle-\langle\be^{r}+\frac{1}{\alpha^r}(\bx^{r+1}-\bx^{r})+\bm g,{\bm y}-{\bm x^{r+1}}\rangle\nonumber\\
&\geq f_{1}({\bm x^{r+1}})+f_{2}({\bm x^{r+1}})-\frac{L}{2}\|\bx^{r+1}-\bx^{r}\|^{2}-\langle\be^{r}+\frac{1}{\alpha^r}(\bx^{r+1}-\bx^{r})+\bm g,{\bm y}-{\bm x^{r+1}}\rangle\nonumber\\
&= F(\bx^{r+1})-\frac{L}{2}\|\bx^{r+1}-\bx^{r}\|^{2}-\langle\be^{r}+\frac{1}{\alpha^r}(\bx^{r+1}-\bx^{r})+\bm g,{\bm y}-{\bm x^{r+1}}\rangle.\nonumber
\end{align}
Specializing $\by=\bx^{r}$, and using the Cauchy-Schwartz inequality, we get
\begin{align}
F(\bx^{r})-F(\bx^{r+1})\geq \frac{2-2\alpha^r\kappa-\alpha^r L}{2\alpha^{r}}\|\bx^{r+1}-\bx^{r}\|^2.\nonumber
\end{align}
Moreover, we know that
\[
\frac{2-2\alpha^r\kappa-\alpha^r L}{2\alpha^{r}}\geq \frac{2-2\bar\alpha\kappa-\bar\alpha L}{2\bar\alpha}>0, \;\;\forall\ r,
\]
which further implies the desired result.
\end{proof}

Finally, we state the following lemma which is needed to prove the linear convergence of the APS method.
Its proof is relegated to the Appendix.

\begin{lemma}\label{lemma:lemma1}
For $\alpha >0$, we have
\begin{enumerate}
\item The function $\alpha\|\tilde{\nabla}f(\bx,\alpha)\|$ is monotonically increasing with $\alpha$.
\item The function $\|\tilde{\nabla}f(\bx,\alpha)\|$ is monotonically decreasing with $\alpha$.
\end{enumerate}
\end{lemma}

Now we are ready to state and prove the linear convergence of APS class of algorithms.

\begin{theorem}\label{theorem:LinCon}
Assume that the error bound property holds for problem \eqref{eq:main}. Consider an APS algorithm
defined by \eqref{eq:APS}-\eqref{eq:ErrCond} for some $\kappa>0$ and some stepsizes $\alpha^r$ satisfying the sufficient decrease condition \eqref{eq:SuffD}.
Then the sequence of iterates generated by the algorithm converges R-linearly to an optimal solution of \eqref{eq:main}.
\end{theorem}

\begin{proof}
The sufficient decrease condition \eqref{eq:SuffD} implies
\begin{equation}
\| {\bm x}^{r+1}-{\bm x}^{r} \|^{2} \to 0 .\nonumber
\end{equation}

Moreover, we have
\begin{align}
\| {\bm x}^{r}-{\rm prox}_{\alpha^{r}f_{1}}&[{\bm x}^{r}-\alpha^{r}\nabla f_{2}({\bm x}^{r})]\| \nonumber \\
&\le  \| {\bm x}^{r}-{\bm x}^{r+1}\| + \| {\bm x}^{r+1}-{\rm prox}_{\alpha^{r}f_{1}}[{\bm x}^{r}-\alpha^{r}\nabla f_{2}({\bm x}^{r})]\| \nonumber \\
&\le \| {\bm x}^{r}-{\bm x}^{r+1}\| + \alpha^r\| {\bm e}^{r}\| \nonumber \\
&\le (\bar\alpha\kappa +1)\| {\bm x}^{r}-{\bm x}^{r+1}\|, \nonumber
\end{align}
where the second inequality is due to the non-expansiveness property of the proximity operator.

Since $\alpha^{r}\ge \underline \alpha$ for all $r$, we obtain from Lemma~\ref{lemma:lemma1} that
\begin{align}
\psi(\bx^{r})&=\| {\bm x}^{r}-{\rm prox}_{f_{1}}[{\bm x}^{r}-\nabla f_{2}({\bm x}^{r})]\| \nonumber \\
& \le \frac{1}{\min\{1,\underline\alpha\}} \| {\bm x}^{r}-{\rm prox}_{\alpha^{r} f_{1} }[{\bm x}^{r}-\alpha^{r}\nabla f_{2}({\bm x}^{r})]\| \nonumber \\
& \le  \frac{1}{\min\{1,\underline\alpha\}} \| {\bm x}^{r}-{\bm x}^{r+1}\| \rightarrow 0\nonumber.
\end{align}
 Hence,
\begin{align}
\psi(\bx^{r}) \to 0. \nonumber
\end{align}
Using the local error bound condition, for sufficiently large $r$, there exists a constant $\tau$ such that
\begin{align}
\varphi({\bm x}^{r})\leq \tau \psi(\bx^{r})\rightarrow 0,
\end{align}
which further implies $\varphi(\bx^{r})\to 0$. Notice that by Lemma~\ref{lm:4.1}, the cost-to-go estimate \eqref{eq:Cost2Go} holds, which together with the fact that $\varphi(\bx^{r})\to 0$ further implies
\[
F({\bm x}^{r})\to F^{\ast}.
\]
Now we use the local error bound condition together with the cost-to-go estimate to obtain
\begin{align}
F({\bm x}^{r+1})-F^{\ast} &\le c_{2}\,(\varphi^2({\bm x}^{r})+\| {\bm x}^{r}-{\bm x}^{r+1}\|^{2})\nonumber \\
&\le c_{2}\,({\tau}\| {\bm x}^{r}-{\rm prox}_{f_{1}}[{\bm x}^{r}-\nabla f_{2}({\bm x}^{r})]\|^{2}+ \| {\bm x}^{r}-{\bm x}^{r+1}\|^{2})\nonumber \\
&\le \frac{c_{2}\tau}{\min\{ 1, {\underline\alpha}^{2}\}}(\| {\bm x}^{r}-{\rm prox}_{\alpha^{r}f_{1}}[{\bm x}^{r}-\alpha^{r}\nabla f_{2}({\bm x}^{r})]\|^{2})+c_{2}\| {\bm x}^{r}-{\bm x}^{r+1}\|^{2}.\nonumber
\end{align}
Next we use \eqref{eq:ErrCond} and non-expansiveness of the proximity operator to bound
\begin{align}
\| {\bm x}^{r}-{\rm prox}_{\alpha^{r}f_{1}}[{\bm x}^{r}-\alpha^{r}\nabla f_{2}({\bm x}^{r})]\|^{2} ~&\leq ~2(\| {\bm x}^{r}-\bx^{r+1}\|^{2}+\| \bx^{r+1}-{\rm prox}_{\alpha^{r}f_{1}}[{\bm x}^{r}-\alpha^{r}\nabla f_{2}({\bm x}^{r})]\|^{2})\nonumber\\
&\leq ~ 2(\bar\alpha^2\kappa^{2}+1)\|\bx^{r+1}-\bx^{r}\|^{2}.\nonumber
\end{align}
This further implies
\begin{align}
F(\bx^{r+1})-F^{*}&\le c_{2}\left(\frac{2\tau(1+\bar\alpha^2\kappa^{2})}{\min\{ 1, {\underline\alpha}^{2}\}}+1\right)\| {\bm x}^{r}-{\bm x}^{r+1}\|^{2}\nonumber \\
&\le  \frac{c_{2}}{c_{1}}\left(\frac{2\tau(1+\bar\alpha^2\kappa^{2})}{\min\{ 1, {\underline\alpha}^{2}\}}+1\right)(F({\bm x}^{r})-F({\bm x}^{r+1})),\nonumber
\end{align}
where the last step is due to the sufficient decrease condition \eqref{eq:SuffD}. This establishes the Q-linear convergence of $F({\bm x}^{r})\to F^{\ast}$. In light of the sufficient decrease condition \eqref{eq:SuffD}, this further implies the R-linear convergence of $\{{\bm x}^{r}\}$ to an optimal solution. The proof of Theorem \ref{theorem:LinCon} is complete.
\end{proof}

An immediate corollary of Theorem~\ref{theorem:LinCon} is that any APS algorithm with stepsizes generated according to \eqref{eq:stepsize} converges linearly, provided that the local error bound condition \eqref{eq:ErrB} holds.

\section{Block Coordinate Descent Method as an APS method}\label{sec:BCD}
Let ${\bm x}\in \mathbb{R}^n$ have the block form of ${\bm x}=({\bm x}_{1},{\bm x}_2,...,{\bm x}_{K})'$, where ${\bm x}_{k}\in {\mathbb{R}^{i_{k}} }$ and  $\sum_{k=1}^{K} i_{k}=n$.
Consider the minimization problem \eqref{eq:main} in which $f_{1}$ is separable over the blocks, namely,
\begin{align}
f_{1}(\bm x)=d_{1}(\bm x_{1})+\cdots+d_{K}(\bm x_{K}),\label{eq:seperable}
\end{align}
where $d_{k}, ~ k=1,\cdots,K$ are all convex (but not necessarily smooth) functions. Furthermore, ${\bf X}$ is a closed convex set in $\mathbb{R}^{n}$ which is in the following Cartesian product form
\begin{equation}\label{cartes}
{\bf X}={\bf X}_{1}\times{\bf X}_{2}\times...\times{\bf X}_{K},
\end{equation}
where ${\bf X}_{k}$ is a closed convex subset of ${\mathbb R}^{i_{k}}$. Note that the LASSO problem \eqref{eq:LASSO}, the group LASSO problem \eqref{eq:GLASSO} and the logestic group LASSO problem \eqref{eq:LogReg} admit the decomposition specified by \eqref{eq:seperable} and \eqref{cartes}.


Consider the BCD method whereby after the ${r}$-th iteration, $r\ge 0$,  we choose an index $s\in \{ 1,2,...,K\}$ and compute the new iterate ${\bm x}^{r+1}=({\bm x}_{1}^{r+1},{\bm x}_{2}^{r+1},...,{\bm x}_{K}^{r+1})$ as follows
\begin{align}\label{blcoord}
{\bm x}_{s}^{r+1}&= {\rm argmin}_{{\bm x}_{s}\in {{\bf X}_{s}}} F({\bm x}_{1}^{r},{\bm x}_{2}^{r},..., {\bm x}_{s-1}^{r},{\bm x}_{s},{\bm x}_{s+1}^{r},...,{\bm x}_{K}^{r}) \nonumber \\
{\bm x}_{j}^{r+1}&={\bm x}_{j}^{r},~~j\neq s.
\end{align}
where $({\bm x}_{1}^{r},{\bm x}_{2}^{r},...,{\bm x}_{K}^{r})$ denotes the iterate at $r$-th iteration.
The blocks are chosen cyclically or essentially cyclically
to be updated at every iteration. The essentially cyclic update ensures that there exists an integer $N \ge K$ such that after this many iterations all the blocks are updated at least once.

It is known \cite{Tseng01} that the BCD algorithm with cyclic update or essentially cyclic update converges to the optimal solution for the set of non-smooth optimization problems that the non-smooth part is separable as defined in \eqref{eq:seperable}.
To establish linear convergence of the BCD method, we need the assumption that the smooth part $f_{2}$ is strongly convex in each block, in the sense that there exists a scalar $\gamma \ge0$ such that, for any ${\bm x}=({\bm x}_{1},{\bm x}_{2},...,{\bm x}_{K})\in {\bf X}$ and any $s\in \{ 1, 2,..., K\}$,
\begin{equation}\label{strongconv}
f_{2}({\bm x}_{1}, {\bm x}_{2},...,{\bm x}_{s-1},{\bm x}_{s}+\Delta{\bm x}_{s},{\bm x}_{s+1},...,{\bm x}_{K})-f_{2}({\bm x})-\langle\nabla_{s} f_{2}({\bm x}), \Delta {\bm x}_{s}\rangle\ge \gamma \| \Delta{\bm x}_{s}\|^{2},
\end{equation}
for all feasible $\Delta{\bm x}_{s}\in \mathbb{R}^{i_{s}} $, where $\nabla_{s}f_{2}$ denotes the vector of partial derivatives of $f_{2}$ with respect to the $s$-th block. It is obvious that if the function $f_2$ is block coordinate-wise strongly convex, then the coordinate descent method satisfies the sufficient decrease condition \eqref{eq:SuffD} [cf. Proposition 3.4 in \cite{Luo93}].

For the applications described in Section \ref{seq:Introduction}, the coordinate-wise strong convexity of $f_2$ imposes a mild condition on the linear operator $\bm A$. For example, in LASSO problem \eqref{eq:LASSO}, if each column of $\bm A$ is non-zero and we consider each element in $\bx$ to be a block, then the problem is coordinate-wise strongly convex. Furthermore, for the group LASSO problem, $f_2$ is block coordinate-wise strongly convex if the columns of $\bm A$ corresponding to a block are linearly independent. A similar condition can be derived for the logestic group LASSO problem \eqref{eq:LogReg} to ensure  block coordinate-wise strong convexity. 
%
The following proposition shows that the block coordinate method for the $L_1$ norm minimization problem and the Group LASSO minimization problem is an APS method.

\begin{proposition}\label{prop:BCD} Under the assumptions \eqref{eq:seperable}, \eqref{cartes} and \eqref{strongconv}, the block coordinate descent method with cyclic update is an APS method satisfying \eqref{eq:APS}-\eqref{eq:ErrCond} with some constant $\kappa>0$. 
\end{proposition}

\begin{proof}
Let us define two different iteration counters. The outer iteration index $s$ is a counter for the number of updating cycles of the BCD algorithm, and the inner iteration index $k$ corresponds to the variable block being updated in a given cycle. Thus, at iteration $r=sK+k$ (with $1\leq k \leq K$), the $k$-th variable block is updated in the $s$-th cycle. Throughout the proof, the notation $\bm x^{r}$ means the $r$-th iterate of the BCD algorithm, and  $\bm x^{r}_{k}$ represents the $k$-th block of $r$-th iterate.

For simplicity,
let us assume that there are no constraints. This assumption is not restricting as one can always add the indicator functions of the constraining sets to the objective. Since the feasible set is assumed to have a special structure as in \eqref{cartes}, the separability of the non-smooth objective component will still be preserved after this change.

The optimality condition at the $r$-th iteration for BCD method is,
\begin{align}
\bm{g}+\nabla_{k}f_{2}(\bm x^{r})=\bm 0,
\end{align}
for some $\bm{g}$ in $\partial d_{k}(\bm x^{r}_{k})$ (Note that we assumed $f_{1}(\bx)=d_{1}(\bx_{1})+\cdots+d_{K}(\bx_{K})$). Now in each fixed cycle $s$,  define $r'=Ks$ and the error vector $\bm e^{s}=(\bm e^{s}_{1},\cdots, \bm e^{s}_{K})$, as follows
\begin{align}
\bm e_{k}^{s} =  \bm x_{k}^{r'+k}-\bm x_{k}^{r'}+\nabla_{k}f_{2}(\bm x^{r'})-\nabla_{k}f_{2}(\bm x^{r'+k}), ~\forall~k=1,\cdots, K.
\end{align}
Then, it is obvious that $\bm x^{r'+K}$ generated by BCD can also be derived from the following update rule
\begin{align}
\bm x^{(s+1)K}=\mbox{prox}_{f_{1}}(\bm x^{sK}-\nabla f_{2}(\bm x^{sK})+\bm e^{s}).
\end{align}
Now we can show that $\|\bm e^{s}\|\leq \kappa \|\bm x^{(s+1)K}-\bm x^{sK}\|$ for some $\kappa>0$. Since $f_{2}$ has Lipschitz continuous gradient, it follows that
\begin{align}
\|\bm e^{s}_{k}\| &~\leq \|\bm x_{k}^{r'+k}-\bm x_{k}^{r'} \|+\|\nabla_{k}f_{2}(\bm x^{r'})-\nabla_{k}f_{2}(\bm x^{r'+k})\|\nonumber\\
&\leq ~\|\bm x_{k}^{r'+k}-\bm x_{k}^{r'}\|+L\|\bm x^{r'}-\bm x^{r'+k}\|\nonumber\\
&\leq ~(L+1)\|\bm x^{r'+k}-\bm x^{r'}\|~\leq~ (L+1)\|\bm x^{r'+K}-\bm x^{r'}\|~=~(L+1)\|\bm x^{(s+1)K}-\bm x^{sK}\|,  \nonumber
\end{align}
where the second step is due to the Lipschitz condition on $\nabla f_{2}$ and the last inequality is due to the block coordinate-wise update in the algorithm. This further implies that
\[
\|\bm e^s\|\le K(L+1)\|\bm x^{(s+1)K}-\bm x^{sK}\|,
\]
so that the condition \eqref{eq:ErrCond} holds with $\kappa=K(L+1)$.
\end{proof}

Note that a similar argument can be used to show that the BCD algorithm with essentially cyclic update also lies in the APS framework, with an error term $\bm e$ satisfying \eqref{eq:ErrCond}.

The following result is a direct consequence of Corollary \ref{coro:EB} and Proposition \ref{prop:BCD}.
\begin{proposition}\label{prop:BCDlin} The BCD algorithm (with cyclic or essentially cyclic update) generates a sequence of iterates that converges R-linearly to a solution in $X^{*}$ for the LASSO problem \eqref{eq:LASSO}, the Group LASSO problem \eqref{eq:GLASSO} and the logistic Group LASSO problem \eqref{eq:LogReg}, provided that the objective function is block coordinate-wise strongly convex.
\end{proposition}

To our knowledge this is the first result which shows the linear convergence rate of the BCD algorithm under the local error bound condition. In a closely related work,  the authors of \cite{Tseng09CGD,necoara2013dist} established the linear convergence of the (Block) Coordinate Gradient Descent (abbreviated as CGD) algorithm for solving problem \eqref{eq:main} under the assumptions \eqref{eq:seperable}-\eqref{cartes} and a local error bound condition. The BCD and CGD algorithms both exploit block coordinate-wise updates to solve the problem \eqref{eq:main}. However, unlike BCD which solves the exact subproblem in each iteration, CGD approximates the smooth component $f_{2}$ by a strictly convex quadratic function. Therefore, the analysis given in \cite{Tseng09CGD,necoara2013dist} does not imply the linear convergence of BCD. Another relevant work is \cite{Meisam13-BSUM,mairal2014} which provides the convergence analysis for a general class of inexact BCD methods. However, their analysis does not provide convergence rate guarantees for the BCD algorithm when applied to problem \eqref{eq:main}. Finally, the readers can refer to \cite{hong2013-BCD} for a unified iteration complexity analysis for a family of BCD-type algorithms.

\section{Concluding Remarks}
In this paper we have introduced the class of approximate proximal splitting methods and established its linear convergence under some conditions (sufficient decrease and local error bound).  This general result implies the linear convergence of the BCD algorithm for a class of non-smooth convex problems. As a future work, it will be interesting to generalize the proofs of linear convergence for the APS algorithms to the problems with nuclear norm regularization \cite{Candes2010}, \cite{Yun2009}.


%
\bibliographystyle{siam}
\bibliography{ref}
\vspace{3 mm}
\newpage
\appendix

\noindent\textbf{Proof of Lemma \ref{lemma:lemma1}}\\
We define $h(\alpha)=\|\tilde{\nabla}f(\bx,\alpha)\|$ for any $\alpha > 0$. Hence, the first part of the lemma is to show that $\alpha h(\alpha)$ is increasing with $\alpha$. \\
From the definition of the proximity operator \eqref{eq:prox} and the proximal gradient \eqref{eq:PSgrad}, we have
\[
\alpha h(\alpha) =\left \|\bx - \arg \min_{\by\in X} \left\{\alpha f_{1}(\by)+\frac{1}{2}||\by-(\bx-\alpha \nabla f_{2}(\bx))||^{2}\right\}\right \|.
\]
By the change of variable $\bz \triangleq \by-\bx$, we have
\begin{align}
\alpha h(\alpha)&=\left\| \arg \min_{\bz \in X^{'}}\left\{\alpha f_{1}(\bx+\bz)+\frac{1}{2}\|\bz+\alpha \nabla f_{2}(\bx) \|^{2}\right\} \right \| \label{eq2:lemma43} \nonumber \\
&=\left\| \arg \min_{\bz \in X^{'}}\left\{\alpha f_{1}(\bx+\bz)+\frac{1}{2}\| \bz \|^{2}+\alpha \bz^{T} \nabla f_{2} (\bx) \right \}\right \|,
\nonumber\end{align}
where $X^{'}=\{\bz\mid\bz = \by - \bx\; \text{for some}\;\by \in X  \}$.
Then, we have
\begin{equation}\label{eq3:lemma43}
\alpha h(\alpha) = \left\|\arg \min_{\bz\in X^{'}} \left\{ \alpha g(\bz) + \frac{1}{2} \| \bz \|^{2} \right\}\right \|,
\end{equation}
where $g(\bz)=f_{1}(\bx+\bz)+ \bz^{T} \nabla f_{2}(\bx)$ is a (non-smooth) convex function. Our goal now is to show that if $0<\alpha_{1}<\alpha_{2}$, then
\[
\alpha_{1}h(\alpha_{1})= \|\bz^{\ast}(\alpha_{1})\|\le \|\bz^{\ast}(\alpha_{2})\|=\alpha_{2} h(\alpha_{2})
\]
where $\bz^{\ast}(\alpha)$ denotes the optimal solution of
\begin{equation}\label{eq4:lemma43}
\min_{\bz \in X^{'}} \left\{ \alpha g(\bz) + \frac{1}{2} \| \bz \|^{2} \right\}.
\end{equation}
The optimality of $\bz^{\ast}(\alpha_{1})$ implies that
\[
g(\bz^{\ast}(\alpha_{1}))+\frac{1}{2 \alpha_{1}} \|\bz^{\ast}(\alpha_{1})\|^{2}  \le g(\bz)+\frac{1}{2 \alpha_{1}} \|\bz\|^{2}, \;\; \forall\ \bz \in X^{'}
\]
In particular, when $\bz$ is set to $\bz^{\ast}(\alpha_{2})$, we have
\begin{equation}\nonumber
g(\bz^{\ast}(\alpha_{1}))+\frac{1}{2 \alpha_{1}} \|\bz^{\ast}(\alpha_{1})\|^{2}  \le g(\bz^{\ast}(\alpha_{2}))+\frac{1}{2 \alpha_{1}} \|\bz^{\ast}(\alpha_{2})\|^{2}.
\end{equation}
Similarly, the optimality of $\bz^{\ast}(\alpha_{2})$ implies that
\begin{equation}\nonumber
g(\bz^{\ast}(\alpha_{2}))+\frac{1}{2 \alpha_{2}} \|\bz^{\ast}(\alpha_{2})\|^{2}  \le g(\bz^{\ast}(\alpha_{1}))+\frac{1}{2 \alpha_{2}} \|\bz^{\ast}(\alpha_{1})\|^{2}.
\end{equation}
Adding up the last two inequalities yields
\[
\left(\frac{\alpha_{2}-\alpha_{1}}{2\alpha_{1}\alpha_{2}} \right) \| \bz^{\ast}(\alpha_{1}) \| \le \left( \frac{\alpha_{2}-\alpha_{1}}{2\alpha_{1}\alpha_{2}}\right) \| \bz^{\ast}(\alpha_{2}) \|.
\]
Since $0<\alpha_{1}<\alpha_{2}$, the above inequality implies that $\| \bz^{\ast}(\alpha_{1}) \| \le \| \bz^{\ast}(\alpha_{2}) \| $.  Note that the convexity of $f_{1}$ (or equivalently $g$) was not used in this part.

Next we prove the second part of the lemma which states that $h(\alpha)$ is monotonically decreasing with $\alpha$.
Introducing the new variable $\bu \triangleq \frac{1}{\alpha}\bz$, the equation \eqref{eq3:lemma43} can be rewritten as
\[
h(\alpha) =\left \| \arg \min_{\bu \in X^{''} } \left\{ \alpha g(\alpha \bu) + \frac{1}{2} \alpha^{2} \| \bu\|^{2}  \right \} \right \|
\]
or equivalently,
\begin{equation}\nonumber\label{eq41:lemma43}
h(\alpha) =\left \| \arg \min_{\bu \in X^{''}} \left\{ \frac{1}{\alpha} g(\alpha \bu) + \frac{1}{2} \| \bu\|^{2}  \right \} \right \|
\end{equation}
where $X^{''}=\{\bu\mid\bu = \frac{1}{\alpha} (\by-\bx), \; \text{for some} \; \by \in X\}$. We define $\bu^{\ast}(\alpha)$ as the optimal solution of
\begin{equation}\label{eq41:lemma44}
h(\alpha) = \min_{\bu \in X^{''}} \left\{ \frac{1}{\alpha} g(\alpha \bu) + \frac{1}{2} \| \bu\|^{2}  \right\}.
\end{equation}
It suffices to show that
\[
h(\alpha_{1}) = \| \bu^{\ast}(\alpha_{1}) \| \ge \| \bu^{\ast}(\alpha_{2}) \| =h(\alpha_{2}) ,
\]
 for $0<\alpha_{1}<\alpha_{2}$. The first order optimality condition of \eqref{eq41:lemma44} at $\bu^{\ast}(\alpha)$ implies
\begin{equation}
\bv+\bu^{\ast}(\alpha)=0, \;\; \text{for some }\bv \in \partial g(\alpha \bu^{\ast} (\alpha)),\label{eq42:lemma43}
\end{equation}
where $\partial g(\alpha \bu^{\ast} (\alpha) ) $ is the sub-differential set of the function $g$ at the point $\alpha \bu^{\ast}(\alpha)$. Rewriting \eqref{eq42:lemma43} for $\bu^{\ast}(\alpha_{1})$ and $\bu^{\ast}(\alpha_{2})$, we obtain
\begin{align}
\bv_{1}+\bu^{\ast}(\alpha_{1})&=0, \;\; \text{for some }\bv_{1} \in \partial g(\alpha_{1} \bu^{\ast} (\alpha_{1})), \label{eq5:lemma43}\\
\bv_{2}+\bu^{\ast}(\alpha_{2})&=0, \;\; \text{for some }\bv_{2} \in \partial g(\alpha_{2} \bu^{\ast} (\alpha_{2})).\label{eq6:lemma43}
\end{align}
Since $g$ is a convex function, $\partial g$ is a monotone mapping \cite{roc70}. Therefore, $\bv_{1} \in \partial g(\alpha_{1} \bu^{\ast} (\alpha_{1}))$ and $\bv_{2} \in \partial g(\alpha_{2} \bu^{\ast} (\alpha_{2}))$ imply
\begin{equation}\label{eq7:lemma43}
\langle \alpha_{1}\bu^{\ast}(\alpha_{1})-\alpha_{2}\bu^{\ast}(\alpha_{2}),\bv_{1}-\bv_{2} \rangle \ge 0.
\end{equation}
Combining \eqref{eq7:lemma43} with \eqref{eq5:lemma43} and \eqref{eq6:lemma43} implies that
\[
\langle \alpha_{1}\bu^{\ast}(\alpha_{1})-\alpha_{2}\bu^{\ast}(\alpha_{2}), \bu^{\ast}(\alpha_{2})-\bu^{\ast}(\alpha_{1})\rangle \ge 0.
\]
Define $\bd=\bu^{\ast}(\alpha_{2})-\bu^{\ast}(\alpha_{2})$. Then the above inequality can be written as
\[
\langle\, (\alpha_{1} - \alpha_{2})\bu^{\ast} (\alpha_{1}) -\alpha_{2} \bd , \bd\, \rangle \ge 0
\]
which yields
\[
\alpha_{2} \|\bd\|^{2} \le (\alpha_{1}-\alpha_{2})\langle \bu^{\ast}(\alpha_{1}), \bd \rangle.
\]
Since $\alpha_{1}-\alpha_{2}<0$, we have
\begin{equation}\label{eq8:lemma43}
\langle \bu^{\ast}(\alpha_{1}) , \bd \rangle \le \frac{\alpha_{2}}{\alpha_{1}-\alpha_{2}} \| \bd \|^{2}.
\end{equation}
Now we can write that
\begin{align}
\|\bu^{\ast}(\alpha_{2})\|^{2}&= \|\bu^{\ast}(\alpha_{1})\|^{2} + 2 \langle \bu^{\ast}(\alpha_{1}) ,\bd \rangle + \|\bd\|^{2} \nonumber\\
&\le \| \bu^{\ast}(\alpha_{1}) \|^{2} + \frac{2\alpha_{2}}{\alpha_{1}-\alpha_{2}} \|\bd\|^{2}  +\|\bd\|^{2} \nonumber \\
&= \| \bu^{\ast}(\alpha_{1}) \|^{2}+ \frac{\alpha_{1}+\alpha_{2}}{\alpha_{1}-\alpha_{2}}\|\bd\|^{2} \nonumber \\
&\le \| \bu^{\ast}(\alpha_{1}) \|^{2}, \nonumber
\end{align}
where the first inequality is due to \eqref{eq8:lemma43} and the second inequality is due to the fact that $0<\alpha_{1}<\alpha_{2}$. This completes the proof.


\end{document}